\documentclass[reqno,12pt]{amsart}
\usepackage{amsmath,amsfonts,amsthm,amssymb,indentfirst}
\usepackage[all,poly]{xy} 
\usepackage{color}
\usepackage[pagebackref,colorlinks]{hyperref}
\usepackage{kantlipsum}
\usepackage[all,poly]{xy}
\usepackage{mathrsfs}
\usepackage{enumerate}

\theoremstyle{plain}
\newtheorem{lemma}{Lemma}[section] 
\newtheorem{theorem}[lemma]{Theorem}
\newtheorem{corollary}[lemma]{Corollary}
\newtheorem{proposition}[lemma]{Proposition}

\theoremstyle{definition}
\newtheorem{remark}[lemma]{Remark}
\newtheorem{example}[lemma]{Example}

\newcommand{\Zset}{\mathbb Z}
\newcommand{\Cset}{\mathbb C}

\newcommand{\ol}{\overline}

\newcommand{\so}{\mathbf{s}}
\newcommand{\ra}{\mathbf{r}}

\newcommand{\ann}{\operatorname{ann}}

\title[Annihilator ideals of graph algebras]{Annihilator ideals of graph algebras}

\author{Lia Va\v s}
\address{Department of Mathematics, Saint Joseph's University, Philadelphia, PA 19131, USA}
\email{lvas@sju.edu}

\subjclass{16S88, 16D25, 16D80, 46L55} 

\keywords{Annihilator ideal, Leavitt path algebra, graph $C^*$-algebra, graded, gauge-invariant, lattice}

\begin{document}

\begin{abstract}
If $I$ is a (two-sided) ideal of a ring $R$, we let $\operatorname{ann}_l(I)=\{r\in R\mid rI=0\},$ $\operatorname{ann}_r(I)=\{r\in R\mid Ir=0\},$ and $\operatorname{ann}(I)=\operatorname{ann}_l(I)\cap \operatorname{ann}_r(I)$ be the left, the right and the double annihilators. An ideal $I$ is said to be an annihilator ideal if $I=\operatorname{ann}(J)$ for some ideal $J$ (equivalently, $\operatorname{ann}(\operatorname{ann}(I))=I$). We study annihilator ideals of Leavitt path algebras and graph $C^*$-algebras.  

Let $L_K(E)$ be the Leavitt path algebra of a graph $E$ over a field $K.$ If $I$ is an ideal of $L_K(E),$ it has recently been shown that $\operatorname{ann}(I)$ is a graded ideal (with respect to the natural grading of $L_K(E)$ by $\mathbb Z$). We note that $\operatorname{ann}_l(I)$ and $\operatorname{ann}_r(I)$ are also graded. If $I$ is graded, we show that $\operatorname{ann}_l(I)=\operatorname{ann}_r(I)=\operatorname{ann}(I)$ and describe $\operatorname{ann}(I)$ in terms of the properties of a pair of sets of vertices of $E,$ known as an admissible pair, which naturally corresponds to $I.$ Using such a description, we present properties of $E$ which are equivalent with the requirement that each graded ideal of $L_K(E)$ is an annihilator ideal. We show that the same properties of $E$ are also equivalent with each of the following conditions: (1) The lattice of graded ideals of $L_K(E)$ is a Boolean algebra; (2) Each closed gauge-invariant ideal of $C^*(E)$ is an annihilator ideal; (3) The lattice of closed gauge-invariant ideals of $C^*(E)$ is a Boolean algebra. In addition, we present properties of $E$ which are equivalent with each of the following conditions: (1) Each ideal of $L_K(E)$ is an annihilator ideal; (2) The lattice of ideals of $L_K(E)$ is a Boolean algebra; (3) Each closed ideal of $C^*(E)$ is an annihilator ideal; (4) The lattice of closed ideals of $C^*(E)$ is a Boolean algebra.
\end{abstract}

\maketitle

\section{Introduction}

If $R$ is a ring (associative but not necessarily unital) and $M$ is a left $R$-module, then $\ann_l(M)=\{r\in R\mid rm=0$ for all $m\in M\}$ is a two-sided ideal of $R$ called the {\em left annihilator} of $M.$ Similarly, if $N$ is a right $R$-module, then the ideal $\ann_r(N)=\{r\in R\mid nr=0$ for all $n\in N\}$ is called the {\em right annihilator} of $N.$ If $B$ is both a left and a right $R$-module, $\ann(B)=\{r\in R\mid rb=br=0\mbox{ for any }b\in B\}$ is the {\em annihilator} of $B.$   

Taking the double annihilator of a (two-sided) ideal is a closure operator on the set of ideals of a ring $R$ in the sense that the following three properties hold for any two ideals $I$ and $J$ of $R$. 
\begin{enumerate}[\upshape(1)]
\item $\ann\ann$ is extensive: $I\subseteq \ann(\ann(I)).$
\item $\ann\ann$ is monotone: $I\subseteq J$ implies  $\ann(\ann(I))\subseteq \ann(\ann(J)).$
\item $\ann\ann$ is idempotent: $\ann(\ann(\ann(I)))=\ann(I).$
\end{enumerate} 

The ideals which are closed under this closure operator have often been called {\em annihilator ideals}. This name reflects the fact that $I$ is an annihilator ideal if and only if $I=\ann(J)$ for some ideal $J.$ Annihilator ideals of a graph $C^*$-algebra have recently been studied in \cite{Brown_et_al} and annihilator ideals of a Leavitt path algebra have recently been studied in \cite{Goncalves_Royer_regular_ideals} and \cite{Malazani}. In all three papers mentioned, annihilator ideals have been referred to as {\em regular} ideals. Since the term ``regular'' has multiple other uses both in ring theory and in operator theory, we opt to use the name ``annihilator ideal'' instead of ``regular ideal''. 

By \cite[Theorem 3.3]{Goncalves_Royer_regular_ideals},
if $E$ is a graph and $K$ a field, then the annihilator of any ideal of the Leavitt path algebra $L_K(E)$ is graded with respect to the natural grading of $L_K(E)$ by the group of integers. In fact, the proof of this result shows that the left and the right annihilators of an ideal are also graded (Proposition \ref{annihilators_are_graded}). As a corollary, the left and the right annihilators of a {\em graded} ideal $I$ of $L_K(E)$ are equal and, hence, equal to $\ann(I)$ (Corollary \ref{corollary_on_l_r_ann}).

Each graded ideal $I$ of a Leavitt path algebra  is uniquely determined by a pair $(H, S)$ of subsets of vertices, known as an admissible pair (we review the definition in section \ref{section_prerequisites}) and one writes $I=I(H,S)$ in this case. For a graph $C^*$-algebra, the role of graded ideals is taken over by the closed gauge-invariant ideals. 
If $E$ is row-finite (i.e. has no vertices emitting infinitely many edges), an admissible pair has the form $(H,\emptyset)$ and we write it shorter as $H.$ If $E$ is row-finite, \cite[Proposition 3.4]{Brown_et_al} describes the set of vertices $H^\bot$ such that $\ann(I(H))=I(H^\bot)$ for a closed gauge-invariant ideal $I(H)$ of $C^*(E)$ and \cite[Proposition 4.2]{Malazani} has the algebraic analogue of this result for a graded ideal $I(H)$ of $L_K(E).$ For an admissible pair $(H,S)$ of an arbitrary graph $E,$ we introduce the admissible pair $(H^\bot, S^\bot)$ and show that 
\[\ann(I(H,S))=I(H^\bot, S^\bot)\] 
holds both in $L_K(E)$ (Proposition \ref{annihilator_of_graded_ideal}) and in $C^*(E)$ (Corollary \ref{annihilator_of_gauge_inv_ideal}). Thus, we can define an operator $^{\bot\bot}$ on the set of admissible pairs of the graph which is a closure operator: it is extensive, monotone, and idempotent (Proposition \ref{closure_on_admissible}). 
We say that an admissible pair is {\em reflexive} if it is closed under this operator. 

Searching for the properties of $E$ which are equivalent with the condition that each graded ideal of $L_K(E)$ is an annihilator ideal, we arrive to the following properties: (a) each cycle is either without exits or extreme (informally, a cycle is extreme if it has exits and each exit ``returns'' to the cycle, section \ref{section_prerequisites} contains a precise definition), (b) each infinite emitter is on a cycle, and (c) each infinite path has only finitely many bifurcations with ranges that do not connect back to some vertex of the path. 
Theorem \ref{theorem_char_of_reflexive} shows the equivalence of the conditions (a), (b), (c) with the requirement that each admissible pair is reflexive. Because of this,  we say that the graph $E$ is  {\em all-reflexive} if the conditions (a), (b), and (c) hold. As corollary of Theorem \ref{theorem_char_of_reflexive}, we show that 
the lattice of graded ideals of $L_K(E)$ is a Boolean algebra if and only if $E$ is all-reflexive. 

If $E$ is all-reflexive and without cycles with no exits, we say that $E$ is {\em strongly all-reflexive}. This condition is equivalent with the requirement that each ideal of $L_K(E)$ is an  annihilator ideal as well as with the requirement that the lattice of all ideals is a Boolean algebra (Theorem \ref{theorem_Boolean}). 

If $(H,S)$ is a reflexive admissible pair of $E,$ then $E$ is (strongly) all-reflexive if and only if the quotient graph $E/(H,S)$ and the porcupine graph $P_{(H,S)}$ are (strongly) all-reflexive (Proposition \ref{proposition_quotients}). 
Thus, if $I$ is an annihilator ideal of $L_K(E),$ then each (graded) ideal of $L_K(E)$ is an annihilator ideal if and only if the same holds for each (graded) ideal of $I$ and of $L_K(E)/I$ (Corollary \ref{corollary_quotients}).  

In section \ref{section_C_star}, we turn to the graph $C^*$-algebra $C^*(E)$ of a graph $E$. From our previous results, it follows relatively directly that the condition that $E$ is all-reflexive is equivalent with the condition that each closed gauge-invariant ideal of $C^*(E)$ is an annihilator ideal as well as with the requirement that the lattice of closed gauge-invariant ideals of $C^*(E)$ is a Boolean algebra (Corollary \ref{corollary_reflexive_c_star}). A bit more consideration is needed to show that the condition that $E$ is strongly all-reflexive is equivalent with the condition that each closed ideal of $C^*(E)$ is an annihilator ideal as well as with the requirement that the lattice of all closed ideals is a Boolean algebra (Theorem \ref{theorem_Boolean_c_star}). The graph $C^*$-algebra version of Corollary \ref{corollary_quotients} also holds: if $I$ is a closed gauge-invariant annihilator ideal of $C^*(E),$ then each closed (gauge-invariant) ideal of $C^*(E)$ is an annihilator ideal if and only if the same holds for each closed (gauge-invariant) ideal of $I$ and of $C^*(E)/I$ (Corollary \ref{corollary_quotients_c_star}).

\section{Prerequisites}
\label{section_prerequisites}

\subsection{Graded rings and  \texorpdfstring{$\ast$}{TEXT}-rings prerequisites}
A ring $R$ (not necessarily unital) is {\em graded} by a group $\Gamma$ if $R=\bigoplus_{\gamma\in\Gamma} R_\gamma$ for additive subgroups $R_\gamma$ and $R_\gamma R_\delta\subseteq R_{\gamma\delta}$ for all $\gamma,\delta\in\Gamma.$ The elements of the set $\bigcup_{\gamma\in\Gamma} R_\gamma$ are said to be {\em homogeneous}. A left ideal $I$ of a graded ring $R$ is {\em graded} if $I=\bigoplus_{\gamma\in \Gamma} I\cap R_\gamma.$ Graded right ideals and graded ideals are defined similarly. 

A ring $R$ is an involutive ring or a $*$-ring, if there is an anti-automorphism $*:R\to R$ of order two. If $R$ is also a $K$-algebra for some commutative, involutive ring $K$, then $R$ is a $*$-algebra if $(kx)^*=k^*x^*$ for all $k\in K$ and $x\in R.$ If $R$ is a $*$-ring, then each left $R$-module is a right module with $mr=r^*m$ and $\ann_l(M)^*=\ann_r(M).$ If $I$ is a left ideal, then $I^*$ is a right ideal and 
\[\ann_l(I)^*=\ann_r(I^*).\]

\subsection{Graphs and properties of vertex sets}
If $E$ is a directed graph, we let $E^0$ denote the set of vertices, $E^1$ denote the set of edges, and $\so$ and $\ra$ denote the source and the range maps of $E.$ A {\em sink} of $E$ is a vertex which emits no edges and an {\em infinite emitter} is a vertex which emits infinitely many edges. A vertex of $E$ is {\em regular} if it is not a sink nor an infinite emitter. The graph $E$ is {\em row-finite} if it has no infinite emitters and $E$ is {\em finite} if it has finitely many vertices and edges.  

A {\em path} is a single vertex or a sequence of edges $e_1e_2\ldots e_n$ for some positive integer $n$ such that $\ra(e_i)=\so(e_{i+1})$ for $i=1,\ldots, n-1.$  
The set of vertices on a path $p$ is denoted by $p^0.$  The functions $\so$ and $\ra$ extend to paths naturally. A path $p$ is {\em closed}  if $\so(p)=\ra(p).$ A closed path $p$ is {\em simple} if $\so(p)$ is not the source of any other edge of $p.$ A {\em cycle} is a closed path such that different edges in the path have different sources. A cycle has {\em an exit} if a vertex on the cycle emits an edge outside of the cycle. The graph $E$ has {\em Condition (L)} if each cycle has an exit. The graph $E$ has {\em Condition (K)} if each vertex $v$ which is a source of a closed simple path $p$ is a source of another closed simple path different than $p.$
A cycle $c$ is {\em extreme} if $c$ has exits and for every path $p$ with $\so(p)\in c^0,$ there is a path $q$ such that $\ra(p)=\so(q)$ and $\ra(q)\in c^0.$ Informally, we can think that this second condition states that ``every exit returns''. 

An {\em infinite path} is a sequence of edges $e_1e_2\ldots$ such that $\ra(e_n)=\so(e_{n+1})$ for $n=1,2\ldots.$ Just as for finite paths, we use $\alpha^0$ for the set of vertices of an infinite path $\alpha.$ An infinite path $\alpha$ is {\em strictly decreasing} if no two different vertices of $\alpha^0$ are on the same closed path.

If $u,v\in E^0$ are such that there is a path $p$ with $\so(p)=u$ and $\ra(p)=v$, we write $u\geq v.$ For $V\subseteq E^0,$
the set $T(V)=\{u\in E^0\mid v\geq u$ for some $v\in V\}$ is called the {\em tree} of $V.$ Following \cite{Lia_irreducible}, we use $R(V)$ to denote the set $\{u\in E^0\mid u\geq v$ for some $v\in V\}$ and we call it the {\em root} of $V.$
If $V=\{v\},$ we use $T(v)$ for $T(\{v\})$ and $R(v)$ for $R(\{v\}).$ We note that $\ol V$ is used for $R(V)$ in \cite{Brown_et_al}, \cite{Goncalves_Royer_regular_ideals} and \cite{Malazani}. We write $u\gneq v$ to denote that $u\geq v$ and $v\ngeq u.$ Note that this is equivalent with $R(u)\subsetneq R(v).$ 

\subsection{Leavitt path algebra}
Extend a graph $E$ to the graph with the same vertices and with edges $E^1\cup \{e^*\mid e\in E^1\}$ where the range and source functions are the same as in $E$ for $e\in E^1$ and $\so(e^*)=\ra(e)$  and $\ra(e^*)=\so(e)$ for the added edges. 
If $K$ is any field, the \emph{Leavitt path algebra} $L_K(E)$ of $E$ over $K$ is a free $K$-algebra generated by the set  $E^0\cup E^1\cup\{e^\ast\mid e\in E^1\}$ such that for all vertices $v,w$ and edges $e,f,$

\begin{tabular}{ll}
(V)  $vw =0$ if $v\neq w$ and $vv=v,$ & (E1)  $\so(e)e=e\ra(e)=e,$\\
(E2) $\ra(e)e^\ast=e^\ast\so(e)=e^\ast,$ & (CK1) $e^\ast f=0$ if $e\neq f$ and $e^\ast e=\ra(e),$\\
(CK2) $v=\sum_{e\in \so^{-1}(v)} ee^\ast$ for each regular vertex $v.$ &\\
\end{tabular}

By the first four axioms, each element of $L_K(E)$ is a sum of the form $\sum_{i=1}^n k_ip_iq_i^\ast$ for some $n$, paths $p_i$ and $q_i$, and elements $k_i\in K,$ for $i=1,\ldots,n$ where $v^*=v$ for $v\in E^0$ and $p^*=e_n^*\ldots e_1^*$ for a path $p=e_1\ldots e_n.$ Using this representation, it is direct to see that 
$L_K(E)$ is an involutive ring with 
$\left(\sum_{i=1}^n k_ip_iq_i^\ast\right)^*=\sum_{i=1}^n k_i^*q_ip_i^\ast$ where $k_i\mapsto k_i^*$ is any involution on $K$. In addition, $L_K(E)$ is locally unital (with the finite sums of vertices as the local units), and $L_K(E)$ is unital if and only if $E^0$ is finite in which case the sum of all vertices is the identity.

If we consider $K$ to be trivially graded by $\Zset,$ $L_K(E)$ is naturally graded by $\Zset$ so that the $n$-component $L_K(E)_n$ is the $K$-linear span of the elements $pq^\ast$ for paths $p, q$ with $|p|-|q|=n$ where $|p|$ denotes the length of a path $p.$ 

\subsection{Graded ideals of a Leavitt path algebra}
A subset $H$ of $E^0$ is said to be {\em hereditary} if $\ra(p)\in H$ for any path $p$ such that $\so(p)\in H.$ The set $H$ is {\em saturated} if $v\in H$ for any regular vertex $v$ such that $\ra(\so^{-1}(v))\subseteq H.$ For every $V\subseteq E^0,$ there is the smallest hereditary and saturated set which contains $V,$ called the hereditary and saturated closure of $V$ (see \cite[Lemma 2.0.7]{LPA_book}).  

If $H$ is hereditary and saturated, a {\em breaking vertex} of $H$ is an element of the set 
\[B_H=\{v\in E^0-H\,|\, v\mbox{ is an infinite emitter and }\so^{-1}(v)\cap \ra^{-1}(E^0-H)\mbox{ is nonempty and finite}\}.\]
For each $v\in B_H,$ let $v^H$ stands for $v-\sum ee^*$ where the sum is taken over $e\in \so^{-1}(v)\cap \ra^{-1}(E^0-H).$

An {\em admissible pair} is a pair $(H, S)$ where $H\subseteq E^0$ is hereditary and saturated and $S\subseteq B_H.$ If $E$ is a row-finite graph, we shorten the notation $(H,\emptyset)$ to only $H.$
For an admissible pair $(H,S)$, let $I(H,S)$ denote the ideal generated by the elements $H\cup \{v^H \,|\, v\in S \}.$ 
The ideal $I(H,S)$ is graded since it is generated by homogeneous elements and it is the $K$-linear span of the elements $pq^*$ for paths $p,q$ with $\ra(p)=\ra(q)\in H$ and the elements $pv^Hq^*$ for paths $p,q$ with $\ra(p)=\ra(q)=v\in S$ (see \cite[Lemma 5.6]{Tomforde}). The converse holds as well: for a graded ideal $I$, the set $H=I\cap E^0$ is hereditary and saturated and $S=\{v\in B_H\mid v^H\in I\}$ is such that $I=I(H,S)$ (\cite[Theorem 5.7]{Tomforde}, also \cite[Theorem 2.5.8]{LPA_book}). The set of admissible pairs is a lattice for the relation 
\[(H,S)\leq (K, T) \mbox{ if }H\subseteq K\mbox{ and }S\subseteq K\cup T\]
(see \cite[Proposition 2.5.6]{LPA_book} for the meet and the join of this lattice). The correspondence $(H,S)\mapsto I(H,S)$ is a lattice isomorphism of this lattice and the lattice of graded ideals. 

Each admissible pair $(H,S)$ gives rise to the quotient graph $E/(H,S),$ defined as below.
\[
\begin{array}{l}
(E/(H,S))^0=E^0-H \cup\{v'\mid v\in B_H-S\} \\
(E/(H,S))^1=\{e\in E^1\mid \ra(e)\notin H\}\cup\{e'\mid e\in E^1\mbox{ and }\ra(e)\in B_H-S\} 
\end{array}
\]
with $\so$ and $\ra$ the same as on $E^1$ for $e\in E^1 \cap (E/(H,S))^1$ and $\so(e')=\so(e),$ $\ra(e')=\ra(e)'$ for $e'\in (E/(H,S))^1.$ The  quotient algebra $L_K(E)/I(H,S)$ is graded isomorphic to $L_K(E/(H,S))$ (see \cite[Theorem 5.7]{Tomforde}).

Each  admissible pair $(H,S)$ also gives rise to the porcupine graph $P_{(H,S)}$ such that its Leavitt path algebra is graded isomorphic to the ideal $I(H,S)$ (see \cite[Theorem 3.3]{Lia_porcupine}). 
The graph $P_{(H,S)}$ is defined as follows. Let
\[
\begin{array}{l}
F_1(H,S)=\{e_1\ldots e_n\mbox{ is a path of }E\mid \ra(e_n)\in H, \so(e_n)\notin H\cup S\}, \\
F_2(H,S)=\{p\mbox{ is a path of }E\mid \ra(p)\in S,\; |p|>0\}.
\end{array}\]
For each $e\in (F_1(H,S)\cup F_2(H,S))\cap E^1,$ let $w^e$ be a new vertex and $f^e$ a new edge such that $\so(f^e)=w^e$ and $\ra(f^e)=\ra(e).$
Continue this process inductively as follows. 
For each path $p=eq$ where $q\in F_1(H,S)\cup F_2(H,S)$ and $|q|\geq 1,$ add a new vertex $w^p$ and a new edge $f^p$ such that $\so(f^p)=w^p$ and $\ra(f^p)=w^q.$ The set of vertices of $P_{(H,S)}$ is 
\[H\cup S\cup \{w^p \mid p\in F_1(H,S)\cup F_2(H,S)\}.\]
The set of edges of $P_{(H,S)}$  is
\[\{e\in E^1\,|\, \so(e)\in H\}\cup \{e\in E^1\,|\, \so(e)\in S, \ra(e)\in H\}\cup \{f^p\mid p\in F_1(H,S)\cup F_2(H,S)\}\]
The $\so$ and $\ra$ maps of $P_{(H,S)}$ are the same as in $E$ for the common edges and they are defined as above for the new edges.

The following result describes the generators of an ideal which is not necessarily graded.  

\begin{theorem} \cite[Theorem 4]{Ranga_ideals} and \cite[Proposition 2.8.5]{LPA_book}
Let $I$ be a nontrivial ideal of $L_K(E)$ and let $H=I\cap E^0,$ and $S=\{v\in B_H\mid v^H\in I\}.$
The ideal $I$ is generated by $H\cup \{v^H\mid v\in S\}\cup Y$ where $Y$ is a set of mutually orthogonal elements of the form $u + \sum_{i=1}^n k_i c^{m_i},$  $m_i$ are positive integers, $k_i\in K$ are such that at least one is nonzero, $c$ is a (unique) cycle such that its image in $E/(H,S)$ is a cycle without exits, and $u=\so(c).$
The ideal $I$ is nongraded if and only if $Y$ is nonempty.
\label{theorem_nongraded_ideals}
\end{theorem}

\section{Annihilator ideals of Leavitt path algebras}

In the rest of the paper, $E$ is an arbitrary graph, $K$ is a field, and $L_K(E)$ is the Leavitt path algebra of $E$ over $K.$

\subsection{Left and right annihilators are graded}
By \cite[Theorem 3.3]{Goncalves_Royer_regular_ideals}, $\ann(I)$ is a graded ideal for any ideal $I$ of a Leavitt path algebra. In fact, the proof of \cite[Theorem 3.3]{Goncalves_Royer_regular_ideals} shows that $\ann_l(I)$ and $\ann_r(I)$ are {\em also} graded ideals. The proof of the next proposition contains some more details.  

\begin{proposition}
If $I$ is an ideal of $L_K(E),$ then the ideals $\ann_l(I)$ and $\ann_r(I)$ are graded. 
\label{annihilators_are_graded} 
\end{proposition}
\begin{proof}
If $I$ is trivial, the claim trivially holds since both one-sided annihilators of $I$ are equal to $L_K(E).$ Thus, assume that $I$ is nontrivial. Let $A=(I\cap E^0)\cup\{v\in B_H\mid v^H\in I\})\cup Y$ where $Y$ is nonempty if and only if $I$ is not graded and the precise definition of $Y$ is given in Theorem \ref{theorem_nongraded_ideals}. If 
\[V=\{x\in L_K(E)\mid xa=0\mbox{ for all }a\in A\},\]
then $V$ is a left ideal of $L_K(E).$ By the proof of \cite[Theorem 3.3]{Goncalves_Royer_regular_ideals}, $V$ is graded. Let 
\[W=\{x\in L_K(E)\mid xra=0\mbox{ for all }r\in L_K(E)\mbox{ and all }a\in A\}.\]
Then $W$ is a (two-sided) ideal of $L_K(E).$  By the proof of \cite[Theorem 3.3]{Goncalves_Royer_regular_ideals}, $W$ is graded. Thus, $V\cap W$ is a graded left ideal of $L_K(E)$ (in fact, it is $W,$ so it is also a right ideal, see Remark \ref{remark_on_annihilators}). 

We claim that $\ann_l(I)=V\cap W.$ Indeed, since both $A$ and $L_K(E)A$ are contained in $I,$ we have that $\ann_l(I)$ is contained in $V$ and in $W,$ so it is contained in $V\cap W.$ To show the converse, let $x\in V\cap W$ and $r\in I.$ By Theorem \ref{theorem_nongraded_ideals}, $r=\sum_{i=1}^n s_ia_it_i$ for some positive integer $n,$ $s_i, t_i\in L_K(E),$ and $a_i\in A$ for $i=1,\ldots, n.$ Since $x\in V\cap W$, $xs_ia_i=0$ for all $i=1,\ldots, n.$ Hence, $xr=0.$

This shows that $\ann_l(I)$ is graded and the proof for $\ann_r(I)$ follows by symmetry. 
\end{proof}

\begin{remark} Since $L_K(E)$ has local units, note that $V$ and $W$ in the above proof are such that $W\subseteq V.$ Indeed, if $x\in W$ and $a\in A,$ let $u$ be a local unit for $x$ (so $xu=ux=x$). Then $xa=xua=0$ since $x\in W,$ so $x\in V.$ The same argument shows that $W'\subseteq V'$ where $V'$ and $W'$ are the sets from the proof of \cite[Theorem 3.3]{Goncalves_Royer_regular_ideals} ($V'$ and $W'$ are the ``right versions'' of $V$ and $W$). Thus, $\ann_l(I)=W,$ $\ann_r(I)=W'$ and $\ann(I)=W\cap W'.$   
\label{remark_on_annihilators}
\end{remark}

If $M$ is a left module, a right module, or a bimodule, then $\ann_l(M),$ $\ann_{r}(M),$ or $\ann(M)$ may not be graded ideals. To illustrate this, note that $E$ has Condition (K) if and only if each ideal is graded (see \cite[Proposition 2.9.9]{LPA_book} or note that this follows also from Theorem \ref{theorem_nongraded_ideals}). So, let $E$ be any graph which does not satisfy Condition (K) (e.g. $\xymatrix{\bullet^v\ar@(ur, dr)^e}$) and let $I$ be an ideal of $L_K(E)$ which is not graded (for example, the ideal generated by $v+e$). Since $L_K(E)$ is a locally unital ring, we have that $I=\ann_l(L_K(E)/I)=\ann_r(L_K(E)/I)=\ann(L_K(E)/I).$ So, as $I$ is not graded, neither of the three annihilators is graded. 

Corollary \ref{corollary_on_l_r_ann} enables us to drop the subscripts $l$ and $r$ for the annihilators of graded ideals. 

\begin{corollary}
If $I$ is a {\em graded} ideal of $L_K(E)$, then 
$\;\ann_l(I)=\ann_r(I)=\ann(I).$
\label{corollary_on_l_r_ann}
\end{corollary}
\begin{proof}
If $I$ is a graded ideal, then $I^*=I$ (see \cite[Lemma 5.6]{Tomforde}). By Proposition \ref{annihilators_are_graded}, $\ann_l(I)$ and $\ann_r(I)$ are graded, so they are also $*$-invariant. 
Hence, $\ann_l(I)=\ann_l(I)^*=\ann_l(I^*)^*=\ann_r(I).$ This also implies that $\ann(I)=\ann_l(I)=\ann_r(I).$
\end{proof}

\subsection{Annihilator ideals via the admissible pairs}
If $V$ is any set of vertices, let 
\[V^\bot=E^0-R(V)\]
where $R(V)$ is the root of $V.$ It is direct to check that $V^\bot$ is hereditary and saturated. We note that $V^\bot$ corresponds to $V'$ from \cite{Malazani} and $E^0-\ol V$ from \cite{Brown_et_al} and \cite{Goncalves_Royer_regular_ideals}. We use the notation $^\bot$ to emphasize the analogy to taking the orthogonal complements in a Hilbert space.

If $E$ is a row-finite graph and $H$ is a hereditary and saturated set, then  $\ann(I(H))=I(H^\bot)$ by \cite[Proposition 4.2]{Malazani}. Proposition \ref{annihilator_of_graded_ideal} generalizes this result for graphs which are not necessarily row-finite. Before this proposition, we show a short lemma. 

\begin{lemma}
If $v$ is a vertex of $E$ such that $v\in B_{H^\bot},$ then 
$v^{H^\bot}=v-\sum ee^*$ where the sum is taken over the elements $e$ of the nonempty and finite set $ \so^{-1}(v)\cap \ra^{-1}(R(H)).$ We also have that  $v\in R(H)-H.$
\label{lemma_on_B_H_bot}
\end{lemma}
\begin{proof}
Note that $E^0-H^\bot=R(H)$ since $H^\bot=E^0-R(H).$ Hence, the formula for $v^{H^\bot}$ follows directly from the definition of a breaking vertex. Since the set $ \so^{-1}(v)\cap \ra^{-1}(R(H))$ is nonempty, 
$v\in R(H).$ Assuming that $v\in H$ implies that $ \ra(\so^{-1}(v))\subseteq H\subseteq R(H).$ As $v$ emits only finitely many edges to $R(H),$ this cannot happen and so $v\notin H.$ Hence, $v\in R(H)-H.$ 
\end{proof}

\begin{proposition}
If $(H,S)$ is an admissible pair of $E$, then 
\[\ann(I(H,S))=I(H^\bot, S^\bot)\]
where $H^\bot=E^0-R(H)$ and $S^\bot=B_{H^\bot}-S.$
\label{annihilator_of_graded_ideal}
\end{proposition}
\begin{proof}
Since both $\ann(I(H,S))$ and $I(H^\bot, S^\bot)$ are graded ideals, to show their equality it is sufficient to show that $H^\bot=\ann(I(H,S))\cap E^0$ and that 
$S^\bot=\{v\in B_{H^\bot}\mid v^{H^\bot}\in \ann(I(H,S))\}.$

Let $v\in H^\bot.$ To show that $vx=0$ for all $x\in I(H,S),$ it is sufficient to show that $v$ annihilates all the elements that generate $I(H,S)$ as a $K$-vector space. So, let $p$ and $q$ be paths such that $\ra(p)=\ra(q)\in H.$ Since $v\notin R(H)$ and $\so(p)\in R(H),$ $vp=0$ and so $vpq^*=0.$
If $w\in S,$ then $w\in R(H)$ by the definition of $B_H.$ As $v\notin R(H),$ $v\neq w$ and so $vw^H=0$ and $vpw^Hq^*=0$ for all paths $p$ and $q$ with $\ra(p)=\ra(q)=w.$ 

Conversely, if $v\in \ann(I(H,S))\cap E^0,$ then $vp=0$ for every path with its range in $H.$ This shows that $v$ cannot be the source of any such path, so $v\notin R(H).$ Hence, $v\in H^\bot.$

This shows that $H^\bot=\ann(I(H,S))\cap E^0.$ Let us show that 
$S^\bot=\{v\in B_{H^\bot}\mid v^{H^\bot}\in \ann(I(H,S))\}.$
Let $v\in S^\bot=B_{H^\bot}-S$ first. We claim that $v^{H^\bot}$ annihilates any $K$-space generator of $I(H,S).$ 
If $p$ and $q$ are paths such that $\ra(p)=\ra(q)\in H,$ then $v^{H^\bot}pq^*=(v-\sum ee^*)pq^*$ where the sum is taken over the nonempty and finite set $\so^{-1}(v)\cap \ra^{-1}(R(H))$ by Lemma \ref{lemma_on_B_H_bot}.

If $v\neq \so(p),$ then $vp=0$ and $ee^*p=0$ for every $e\in \so^{-1}(v)\cap \ra^{-1}(R(H)),$ so the claim trivially holds. If $v=\so(p),$ then $p$ is of nonzero length since $\ra(p)\in H$ and $\so(p)\in R(H)-H$ by Lemma \ref{lemma_on_B_H_bot}. So, one of the edges, say $e,$ from $\so^{-1}(v)\cap \ra^{-1}(R(H))$ is the first edge of $p$ and $v^{H^\bot}p=p-ee^*p=p-p=0,$ so the claim holds again. 

If $w\in S,$ then $v\neq w$ since $v\notin S.$ So, $v^{H^\bot}w^H=0.$ If $p$ and $q$ are paths such that $\ra(p)=\ra(q)=w,$ and the length of $p$ is larger than zero, then $v^{H^\bot}p=0,$ so $v^{H^\bot}pw^Hq^*=0.$

Conversely, let $v\in B_{H^\bot}$ be such that $v^{H^\bot}\in\ann(I(H,S)).$ Since $v^{H^\bot}$ annihilates the elements of $I(H,S),$ $v^{H^\bot}w^H=0$ for every $w\in S.$ As $w$ emits infinitely many edges to $H$ and $v$ emits only finitely many edges to $R(H),$ $v\neq w$ for every $w\in S.$ Thus, $v\notin S.$ Hence, $v\in B_{H^\bot}-S.$ 
\end{proof}

If $I$ is an ideal which is not graded, $H=I\cap E^0$ and $S=\{v\in B_H\mid v^H\in I\},$ then the graded ideals $\ann_l(I)$ or $\ann_r(I)$ may be strictly contained in $\ann(I(H,S))=I(H^\bot, S^\bot)$ as the next example shows.  

\begin{example}
Let $E$ be the graph $\xymatrix{\bullet^u\ar@(ul,dl)_e   & \bullet^v\ar[l]_f }$ and let $I$ be the ideal generated by $u+e.$ By Theorem \ref{theorem_nongraded_ideals}, $I$ is not graded. 
We have that $H=I\cap E^0=\emptyset,$ so $H^\bot=E^0$ and $I(H^\bot)=L_K(E).$ On the other hand, $\ann_l(I)$ is strictly contained in $L_K(E)$ since $u(u+e)=u+e\neq 0.$
In fact, $v\notin \ann_l(I)$ as $vf(u+e)=f+fe\neq 0.$ Since $\ann_l(I)$ is graded by Proposition \ref{annihilators_are_graded} and $\ann_l(I)\cap E^0=\emptyset$, we have that $\ann_l(I)=\{0\}.$ \label{example_not_graded}
\end{example}

\subsection{The operator  \texorpdfstring{$({^{^{}}\hskip.2cm})^{\bot\bot}$}{TEXT}}
Next, we show that the operator $({^{^{}}\hskip.2cm})^{\bot\bot}$ is a closure operator on the set of admissible pairs of $E.$  

\begin{proposition} 
Considered as an operator on the set of admissible pairs of $E,$ the operator $({^{^{^{}}}\hskip.2cm})^\bot$ is decreasing (i.e. $(H, S)\leq (G, T)$ implies $(H^\bot, S^\bot)\geq (G^\bot, T^\bot)$ for any two admissible pairs $(H,S)$ and $(G,T)).$ 

The operator $({^{^{}}\hskip.2cm})^{\bot\bot}$ is the closure operator on the set of admissible pairs, i.e. the following properties hold for any two admissible pairs $(H,S)$ and $(G,T).$
\begin{enumerate}[\upshape(1)]
\item $({^{^{}}\hskip.2cm})^{\bot\bot}$ is extensive: $(H,S)\leq (H^{\bot\bot}, S^{\bot\bot}).$
\item $({^{^{}}\hskip.2cm})^{\bot\bot}$ is monotone: $(H, S)\leq (G, T)$ implies  $(H^{\bot\bot}, S^{\bot\bot})\leq(G^{\bot\bot}, T^{\bot\bot}).$
\item $({^{^{}}\hskip.2cm})^{\bot\bot}$ is idempotent: $(H^{\bot\bot\bot}, S^{\bot\bot\bot})=(H^\bot, S^\bot).$
\end{enumerate} 
\label{closure_on_admissible}
\end{proposition}
\begin{proof}
To show the first sentence of the proposition, note that $H\subseteq G$ implies that $R(H)\subseteq R(G)$ and so $H^\bot=E^0-R(H)\supseteq E^0-R(G)=G^\bot.$ 
Next, we show that $H\subseteq G$ and $S\subseteq G\cup T$ imply $T^\bot\subseteq H^\bot\cup S^\bot.$
If $v\in T^\bot,$ then $v$ emits infinitely many edges with ranges in $G^\bot\subseteq H^\bot$ and nonzero and finitely many with ranges in $R(G).$ 
We consider the case when none of those with ranges in $R(G)$ are in $R(H)$ and the case when some (hence only finitely many) of those with ranges in $R(G)$ are in $R(H).$  In the first case, each edge $v$ emits has the range in $H^\bot,$ so $v$ is in $H^\bot$ also (otherwise at least one edge $v$ emits would have to have the range in $R(H)$). Hence, $v\in H^\bot\cup S^\bot.$ In the second case, $v$ is in $B_{H^\bot}.$ Note that $v\in T^\bot\subseteq B_{G^\bot}$ implies that $v$ is not in $G$ and, as $v\in T^\bot,$ 
$v$ is not in $T.$ The relation $S\subseteq K\cup T$ implies that $v\notin S.$ So, $v$ is in $B_{H^\bot}-S=S^\bot\subseteq H^\bot\cup S^\bot.$ 

This shows that  $({^{^{^{}}}\hskip.2cm})^\bot$ is decreasing and implies that (2) holds. 

To show (1), let $v\in H.$ Since $H$ is hereditary, this implies that $v\notin R(E^0-R(H)),$ so $v\in E^0-R(E^0-R(H))=H^{\bot\bot}.$ Next, we show that $S\subseteq S^{\bot\bot}\cup H^{\bot\bot}.$
Let $v\in S.$ This implies that $v\notin S^\bot.$ Since $v\in B_H,$ $v$ emits infinitely many edges with ranges in $H\subseteq H^{\bot\bot},$ and nonzero and finitely many with ranges in $E^0-H.$ We consider the case when none of those with ranges in $E^0-H$ are in $E^0-H^{\bot\bot}=R(H^\bot)$ and the case when some (hence only finitely many) of those with ranges in $E^0-H$ are in $R(H^\bot).$ In the first case, each edge $v$ emits has the range in $H^{\bot\bot}.$ This implies that $v\in H^{\bot\bot}$ (otherwise at least one edge $v$ emits would have to have a range in $R(H^\bot)$). So, $v\in H^{\bot\bot}\subseteq S^{\bot\bot}\cup H^{\bot\bot}.$ In the second case,  $v\in B_{H^{\bot\bot}},$ so $v\in  B_{H^{\bot\bot}}-S^\bot=S^{\bot\bot}\subseteq S^{\bot\bot}\cup H^{\bot\bot}.$

To show (3), note that (1) implies that 
$(H^\bot, S^\bot)\leq (H^{\bot\bot\bot}, S^{\bot\bot\bot}).$ Taking $({^{^{^{}}}\hskip.2cm})^\bot$ of the relation in (1) and using that $({^{^{^{}}}\hskip.2cm})^\bot$ is decreasing imply that 
$(H^\bot, S^\bot)\geq (H^{\bot\bot\bot}, S^{\bot\bot\bot}).$  
\end{proof}

We note the following property of $({^{^{}}\hskip.2cm})^{\bot\bot}.$

\begin{proposition}
If $H\subseteq E^0$ is a hereditary and saturated set, then $H^{\bot\bot}$ is the largest hereditary set in $R(H).$ Thus, $R(H)$ is hereditary if and only if $R(H)=H^{\bot\bot}.$ 
\end{proposition}
\begin{proof}
The set $H^{\bot\bot}$ is hereditary (and saturated). It is in $R(H)$ since $H^{\bot\bot}$ and $H^\bot$ are disjoint so $H^{\bot\bot}\subseteq E^0-H^\bot=R(H).$ 
If $G$ is a hereditary set such that $H^{\bot\bot}\subseteq G
\subseteq R(H),$ we claim that $G=H^{\bot\bot}.$ Assume, on the contrary, that there is $v\in G$ such that $v\notin H^{\bot\bot}=E^0-R(H^\bot).$ Hence, $v\in R(H^\bot).$ If $p$ is a path from $v$ to a vertex of $H^\bot,$ then $\ra(p)\in H^\bot\cap G$ since $G$ is hereditary. Thus, $\ra(p)\in H^\bot\cap G\subseteq H^\bot\cap R(H)=E^0-R(H)\cap R(H)=\emptyset$ which is a contradiction.   
\end{proof}

One can easily construct an example of a hereditary and saturated set $H$ such that $R(H)$ is not hereditary. For example, if $E$ is the graph $\xymatrix{\bullet^u & \bullet^v\ar[r]\ar[l] &\bullet^w}$ and $H=\{u\},$ then $R(H)=\{u,v\}$ is not hereditary. Note that $R(H^\bot)=\{v,w\}$ and $H^{\bot\bot}=H=\{u\}\subsetneq R(H).$

\subsection{Reflexive admissible pairs}
We say that an admissible pair $(H,S)$ is {\em reflexive} if $(H,S)=(H^{\bot\bot}, S^{\bot\bot}).$ 
By Proposition \ref{annihilator_of_graded_ideal}, a graded ideal $I=I(H,S)$ is an annihilator ideal if and only if $(H,S)$ is reflexive. In Proposition \ref{proposition_on_reflexive}, we show another  requirement for $(H,S)$ equivalent with $(H,S)$ being reflexive. If $E$ is a row-finite graph, 
\cite[Corollary 4.3]{Malazani} lists another condition similar (and equivalent) to part (3) of Proposition \ref{proposition_on_reflexive}. We show the following lemma before proving Proposition \ref{proposition_on_reflexive}. 

\begin{lemma}
If $(H,S)$ is an admissible pair of $E$ such that $H=H^{\bot\bot}$, then $S^{\bot\bot}=B_H.$  
\label{lemma_for_reflexive}
\end{lemma}
\begin{proof}
Since  $H=H^{\bot\bot},$ we have that $S^{\bot\bot}=B_{H^{\bot\bot}}-S^\bot=B_H-(B_{H^\bot}-S)=(B_H-B_{H^\bot})\cup (B_H\cap S)=(B_H-B_{H^\bot})\cup S.$ Since $B_H\subseteq E^0-B_{H^\bot}$ by definition of breaking vertices, we have that $B_H-B_{H^\bot}=B_H.$ Hence, 
$S^{\bot\bot}=(B_H-B_{H^\bot})\cup S=B_H\cup S=B_H.$
\end{proof}

\begin{proposition}
The following conditions are equivalent for an admissible pair $(H,S)$ of $E.$ 
\begin{enumerate}[\upshape(1)]
\item The ideal $I(H,S)$ is an annihilator ideal of $L_K(E).$ 
 
\item $(H,S)$ is reflexive. 

\item $R(H)-H\subseteq R(H^\bot)$ and $S=B_H.$  
\end{enumerate}
\label{proposition_on_reflexive}
\end{proposition}
\begin{proof}
The first two conditions are equivalent by Propositions \ref{annihilator_of_graded_ideal} and \ref{closure_on_admissible}. 

If (2) holds and if $v\in R(H)-H,$ then $v\notin H=H^{\bot\bot}=E^0-R(H^\bot).$ Hence, $v\in R(H^\bot).$ Condition (2) and Lemma \ref{lemma_for_reflexive} imply that $S=S^{\bot\bot}=B_H$ which shows (3). 

If (3) holds and $v\in H^{\bot\bot},$ then $v\notin R(H^\bot)$ and so $v\notin R(H)-H.$ As $v\notin R(H^\bot),$ $v\notin H^\bot=E^0-R(H)$ which implies that $v\in R(H).$ So, we have that $v\in R(H)$ and $v\notin R(H)-H$ thus $v\in H.$ Hence,  $H=H^{\bot\bot}.$ By Lemma \ref{lemma_for_reflexive},
$S^{\bot\bot}=B_H$ and so $S^{\bot\bot}=B_H=S.$ Hence, (2) holds.  
\end{proof}

Before moving on to the main result, we digress to show Proposition \ref{proposition_condition_L}. This proposition is formulated so that it is a statement on the graph $E$ only, not on an annihilator ideal of $L_K(E)$ or $C^*(E).$ Because of this, it implies both \cite[Proposition 3.10]{Goncalves_Royer_regular_ideals} and \cite[Corollary 3.8]{Brown_et_al} and shows that these results hold without requiring $E$ to be row-finite. 
 
\begin{proposition}
If $E$ satisfies Condition (L) and $(H, S)$ is a reflexive admissible pair of $E,$ then $E/(H,S)$ satisfies Condition (L).   
\label{proposition_condition_L}
\end{proposition}
\begin{proof}
Assume that $E$ has Condition (L), that $(H,S)$ is reflexive, and that $c$ is a cycle without exits in $E/(H,S).$ By Proposition \ref{proposition_on_reflexive}, $S=B_H,$ so there are no vertices added to form the quotient graph. Hence,  $(E/(H,S))^0=E^0-H$ and $c$ is a cycle in $E$ with vertices in $E^0-H.$ As $E$ has Condition (L), $c$ has an exit. Since $c$ has no exits in $E/(H,S),$ the range of any exit of $c$ is in $H.$ This implies that $c^0\subseteq R(H)\cap (E^0-H)=R(H)-H$ and that 
$T(c^0)\subseteq R(H).$ This last relation implies that no vertex of $c$ is in $R(H^\bot).$ On the other hand, by Proposition \ref{proposition_on_reflexive}, $R(H)-H\subseteq R(H^\bot),$ so $c^0\subseteq R(H)-H\subseteq R(H^\bot).$ Thus, we reach a contradiction.
\end{proof}

If $H=\{v\}$ for the graph $\;\;\;\;\xymatrix{\bullet\ar@(lu,ld)  \ar[r] & \bullet^v},$ then  Condition (L) holds on $E$ and fails on $E/H.$  Proposition \ref{proposition_condition_L} shows that this cannot happen if the quotient is taken with respect to a reflexive admissible pair. \cite[Example 3.9]{Brown_et_al} exhibits a row-finite graph $E$ with Condition (L) and a hereditary and saturated set $H$ such that $H$ is not reflexive and $E/H$ satisfies Condition (L).  

\subsection{All-reflexive graphs}
The following example displays three graphs and their  admissible pairs which are not reflexive. 

\begin{example}  
Let $E$ be the graph  $\;\;\;\;\xymatrix{{\bullet}^u\ar@(lu,ld)  \ar[r] & {\bullet}^v}.$ If $H=\{v\},$ then $R(H)=E^0,$ so $R(H)-H=\{u\}\nsubseteq R(H^\bot)=R(\emptyset)=\emptyset.$ Hence, $(H, \emptyset)$ is not reflexive by Proposition \ref{proposition_on_reflexive}.  

Let $E$ be the graph $\xymatrix{\bullet^u \ar@{.} @/_1pc/ [r] _{\mbox{ } } \ar@/_/ [r] \ar [r] \ar@/^/ [r] \ar@/^1pc/ [r] & {\bullet}^v}.$ If $H=\{v\},$ then $B_H=\emptyset$ and $R(H)=E^0,$ so $R(H)-H=\{u\}\nsubseteq R(H^\bot)=R(\emptyset)=\emptyset.$ Hence, $(H,\emptyset)$ is not reflexive by Proposition \ref{proposition_on_reflexive}. 

Lastly, let $E$ be the graph $\xymatrix{   \bullet     & \bullet     & \bullet  &   & \\   \bullet \ar[r]\ar[u] & \bullet \ar[r]  \ar[u] & \bullet \ar[r]\ar[u] & \bullet \ar@{.>}[r] \ar@{.>}[u] &}$ 
and let $H$ be the union of all sinks of $E.$ As $R(H)=E^0,$ $R(H)-H\neq \emptyset$ and  $R(H^\bot)=R(\emptyset)=\emptyset.$ Hence, $(H, \emptyset)$ is not reflexive. 
\label{example_not_reflexive}
\end{example}

Theorem \ref{theorem_char_of_reflexive} states the necessary and sufficient condition on $E$ for each graded ideal of $L_K(E)$ to be an annihilator ideal. It shows that the three graphs from the previous example have a complete list of features which obstruct all admissible pairs from being reflexive: a cycle with exits which is not extreme, an infinite emitter which is not on an extreme cycle, and an infinite path with infinitely many bifurcations not connecting back to the path. 

Before proving the theorem, we show a lemma. 
\begin{lemma}
Let $E$ be a graph such that each cycle is either without exits or extreme and such that each infinite emitter is on a cycle. If $H\subseteq E^0$ is  hereditary and saturated, then no vertex of $R(H)-H$ is on a cycle and each vertex of $R(H)-H$ is regular. Hence, if $(H,S)$ is an admissible pair, then $B_H=S=\emptyset.$   
\label{lemma_on_RH_bez_H} 
\end{lemma}
\begin{proof}
With the assumptions on $E$ of the lemma, assume that a vertex $v$ of $R(H)-H$ is on a cycle $c$. As $H$ is hereditary, $c^0\subseteq R(H)-H$ implies that $c$ has to have an exit. By the assumptions on $E,$ $c$ is extreme. As the vertices in $c^0$ connect only to vertices on the closed paths containing some vertex of $c$ and one of them connects to a hereditary set $H,$ we have to have that $c^0$ is in $H,$ contradicting the assumption that $v\in R(H)-H.$  

If $v$ is an infinite emitter of $R(H),$ it is on a cycle, so it cannot be in $R(H)-H.$ In addition, a vertex $v$ of $R(H)-H$ connects to $H,$ so there is a path originating at $v.$ Thus, $v$ is not a sink. Hence, each vertex of $R(H)-H$ is regular. 

To show the last claim, assume that $(H,S)$ is an admissible pair and that a vertex $v$ is in $B_H.$ This implies that $v\in R(H)-H,$ so $v$ cannot be an infinite emitter. Thus, $B_H=\emptyset$ and so $S=\emptyset.$ 
\end{proof}

\begin{theorem}
The following conditions are equivalent for any graph $E$ and any field $K.$ 
\begin{enumerate}[\upshape(1)]
\item Each graded ideal of $L_K(E)$ is an annihilator ideal. 

\item The following conditions hold for $E.$
\begin{enumerate}[\upshape(a)]
\item Each cycle in $E$ is either without exits or extreme. 
\item Each infinite emitter is on a cycle.  
\item Each infinite path $\alpha$ has only finitely many bifurcations with ranges not in the root of $\alpha^0$ (i.e. $\ra(e)\notin R(\alpha^0)$ for only finitely many edges $e$ with $\so(e)\in \alpha^0$). 
\end{enumerate}
\end{enumerate}
\label{theorem_char_of_reflexive} 
\end{theorem}
\begin{proof}
To show (1) $\Rightarrow$ (2), we show that the negation of (2) implies the existence of a non-reflexive admissible pair $(H,S)$. In that case, $I(H,S)$ is not an annihilator ideal by Proposition \ref{proposition_on_reflexive}.

If (2a) fails, there is a non-extreme cycle $c$ which has an exit with the range not in $R(c^0)$. Hence, 
\begin{center}
$V_a=\{\ra(e)\mid \so(e)$ is on a closed path that contains a vertex of $c$ and $\ra(e)\notin R(c^0)\}$
\end{center}
is not empty. Let $H$ be the hereditary and saturated closure of $V_a.$ Since the vertices of $c$ connect to $\ra(e)$ for each $e$ such that $\ra(e)\in V_a,$ we have that $c^0\subseteq R(H).$ Also, no vertex $v$ of $c$ is in $H$ since $v\in R(c^0).$ For such a vertex $v,$
the range of any path $v$ emits is either in $R(c^0)\subseteq R(H)$ or in $T(V_a)\subseteq H\subseteq R(H),$ so $v\notin R(H^\bot).$ By Proposition \ref{proposition_on_reflexive}, $(H, S)$ is not reflexive for any $S\subseteq B_H.$  

If (2b) fails and (2a) fails also, then we have a non-reflexive pair $(H,S)$ as above. Hence, assume that (2b) fails and that (2a) holds. So, there is an infinite emitter $v$ which is not on a cycle. The set $V_b=\ra(\so^{-1}(v))$ is not empty and, as $v$ is not on a cycle, $\ra(e)\notin R(v)$ for every $e\in \so^{-1}(v).$ Let $H$ be the hereditary and saturated closure of $V_b.$ Then $v\in R(H)$ since $v$ connects to $\ra(e)$ for any $e\in \so^{-1}(v).$ Also, $v\notin H$ since $v$ is not in $T(V_b)$ and $v$ is not regular (see \cite[Lemma 2.0.7]{LPA_book}). As $T(v)\subseteq H \subseteq R(H),$ $v\notin R(H^\bot).$ By Proposition \ref{proposition_on_reflexive}, $(H, S)$ is not reflexive for any $S\subseteq B_H.$ 

If (2c) fails and either (2a) or (2b) fails also, there is a non-reflexive pair obtained as above. Hence, assume that (2c) fails and that (2a) and (2b) hold. Thus, there is an infinite path $\alpha$ with infinitely many bifurcations $e$ whose ranges are not in $R(\alpha^0).$ 
Hence, the set
\begin{center}
$V_c=\{\ra(e)\mid \so(e)\in \alpha^0$ and $\ra(e)\notin R(\alpha^0)\}$
\end{center}
is not empty. Let $H$ be the hereditary and saturated closure of $V_c.$ Since (2a) and (2b) hold, we cannot have a vertex $v\in \alpha^0$ emitting infinitely many edges with ranges in $V_c.$ Hence, for every $v\in \alpha^0,$ there is $u\in \alpha^0$ such that $v\geq u$ and such that $u$ emits an edge with the range in $V_c.$ Thus,
$\alpha^0\subseteq R(H).$ Also, no vertex $v$ of $\alpha$ is in $H$ since $v\in R(\alpha^0).$ By (2a), every path that originates at $v\in \alpha^0$ ends up either in $R(\alpha^0)\subseteq R(H)$ or in $T(V_c)\subseteq H\subseteq R(H).$ So, no range of such path is in $H^\bot$ which shows that $v\notin R(H^\bot).$ Thus, we have an admissible pair $(H,S)$ which is not reflexive for any $S\subseteq B_H$ by Proposition \ref{proposition_on_reflexive}.

This shows that (1) implies (2). To show the converse, assume that $E$ is a graph such that (2) holds and let $I=I(H,S)$ be a graded ideal of $L_K(E).$ By Lemma \ref{lemma_on_RH_bez_H}, $S=\emptyset.$ By Proposition \ref{proposition_on_reflexive}, to show that $I$ is an annihilator ideal, it is sufficient to show that $R(H)-H\subseteq R(H^\bot).$ 
 
For each vertex $v\in R(H)-H,$ let $P(v)$ be the set of paths $p,$ such that $\so(p)=v,$ $\ra(p)\in H,$ and such that no other vertex of $p$ except $\ra(p)$ is in $H.$ Assume that there is a vertex $v_0$ such that $v_0\in R(H)-H$ and $v_0\notin R(H^\bot).$ 
Consider a path in $P(v_0)$ which has the minimal length. Since $v_0\notin H,$ this length is larger than zero, so we can write such path in the form $p_0e_0$ for some path $p_0$ and an edge $e_0.$ Let  $w_0=\so(e_0).$ Since $H$ is saturated, $w_0\notin H$  and $\ra(e_0)\in H,$ $w_0$ emits other edges, by Lemma \ref{lemma_on_RH_bez_H}, finitely many of them, and at least one of them, say $f_0,$ has the range in $R(H)-H.$ Let $v_1=\ra(f_0).$ By Lemma \ref{lemma_on_RH_bez_H},  there are no cycles in $R(H)-H$ and so $v_1\notin R(v_0).$ Thus, $v_0\neq v_1.$ As $v_1\in T(v_0),$ $v_1\notin R(H^\bot).$ 

Repeating the construction for $v_1$ instead of $v_0,$ we obtain a path $p_1$ and an edge $e_1$ such that $p_1e_1$ is a path in $P(v_1)$ of the minimal length. Let $w_1=\so(e_1).$ Repeating the argument for $w_0,$ we obtain $f_1\in \so^{-1}(w_1)$ with $\ra(f_1)\in R(H)-H.$ We let $v_2=\ra(f_1)$ and note that $v_2\notin R(v_0)\cup R(v_1)$ as there are no cycles in $R(H)-H.$ As $v_2\in T(v_0),$ $v_2\notin R(H^\bot).$  

Continuing this construction, we obtain an infinite path $\alpha=p_0f_0p_1f_1p_2f_2\ldots$ which has infinitely many bifurcations $e_0, e_1, e_2,\ldots$ with the ranges in $H.$ These ranges are not in $R(\alpha^0)$ since $H$ is hereditary and every vertex of $\alpha$ is in $R(H)-H.$ However, this contradicts (2c). Hence, a vertex $v_0$ in $R(H)-H$ and not in $R(H^\bot)$ cannot exist. This shows that $R(H)-H\subseteq R(H^\bot).$   
\end{proof}

By Theorem \ref{theorem_char_of_reflexive}, each admissible pair of a graph $E$ is reflexive if and only if $E$ satisfies condition (2) of Theorem \ref{theorem_char_of_reflexive}. In this case, we say that $E$ is {\em all-reflexive}.

If $E$ is finite, the requirements that $E$ is all-reflexive simplify significantly. If $E$ is finite,  then there are no infinite emitters, so condition (2b) of Theorem \ref{theorem_char_of_reflexive} is superfluous. Also, the only infinite path is the one with vertices on cycles, so condition (2a) of Theorem \ref{theorem_char_of_reflexive} implies (2c). Hence, $E$ is all-reflexive if and only if each cycle in $E$ is either without exits or extreme.

\subsection{Conditions for the lattice of ideals to be  a Boolean algebra}\label{subsection_Boolean}
Since $L_K(E) $ is a semiprime ring ($I^2=0$ implies $I=0$ for every ideal $I,$ see \cite[Proposition 2.3.1]{LPA_book}), we have that $I\cap \ann(I)=\{0\}$ for any ideal $I.$ This implies that the lattice of  annihilator ideals is a Boolean algebra (see also \cite[Proposition 4, Section 4.6]{Lambek} or \cite[page 15, Exercise 7]{Steinberg_lattice_ordered}). The meet and the join operations are given by 
\[I\wedge J=I \cap J\;\;\mbox{ and }\;\;I\vee J=\ann(\ann(I)\cap \ann(J))=\ann(\ann(I+J)).\]  

We relate the conditions that $E$ is all-reflexive with the requirement that the lattice of graded ideals of $L_K(E)$ is a Boolean algebra next. 

\begin{theorem}
The following conditions are equivalent for any graph $E$ and any field $K.$  
\begin{enumerate}[\upshape(1)]
\item The graph $E$ is all-reflexive. 
\item The lattice of graded ideals of $L_K(E)$ is a Boolean algebra. 
\end{enumerate}
\label{theorem_Boolean_graded}  
\end{theorem}
\begin{proof}
If (1) holds, and $I$ and $J$ are graded ideals, then $I+J$ is an annihilator ideal by Theorem \ref{theorem_char_of_reflexive} and so $I\vee J=\ann(\ann(I+J))=I+J.$ So, the lattice of graded ideals coincides with the lattice of annihilator ideals. As the latter is a Boolean algebra, the former is a Boolean algebra. 

Assuming (2), we show that condition (1) of Theorem \ref{theorem_char_of_reflexive} holds. Let $I$ be any graded ideal and let $J$ be its complement in the lattice of graded ideals. As $I\cap J=\{0\},$ we have that $yx=xy=0$ for every $x\in I$ and $y\in J.$ This shows that $J\subseteq \ann(I).$ To show $\ann(I)\subseteq J,$ let $r\in \ann(I)$ and let $u$ be a local unit for $r.$ As $I+J=L_K(E),$ $u=x+y$ for some  $x\in I$ and $y\in J.$ So, $r=ru=r(x+y)=0+ry=ry\in J.$
This shows that $J=\ann(I).$ This also shows that $\ann(\ann(I))$ is the complement of $J$. Thus, $\ann(\ann(I))=I.$  
\end{proof}

We say that a graph is {\em strongly all-reflexive} if it is all-reflexive and it has no cycles without exits. If $E$ is finite, $E$ is strongly all-reflexive if and only if each cycle in $E$ is extreme. We relate the strong all-reflexivity of $E$ with properties of the lattice of all ideals of $L_K(E)$. 

\begin{theorem}
The following conditions are equivalent for any graph $E$ and any field $K.$  
\begin{enumerate}[\upshape(1)]
\item Each ideal of $L_K(E)$ is an annihilator ideal. 
\item The graph $E$ is strongly all-reflexive. 
\item The lattice of ideals of $L_K(E)$ is a Boolean algebra. 
\end{enumerate}
\label{theorem_Boolean} 
\end{theorem}
\begin{proof}
If (1) holds, then each graded ideal is an annihilator ideal, so $E$ is all-reflexive by  Theorem \ref{theorem_char_of_reflexive}. Thus, to show (2), it remains to show that there are no cycles without exits. Assume that $c$ is a cycle without exits and let $u=\so(c).$ The ideal $I$ generated by $u+c$ is not graded by Theorem \ref{theorem_nongraded_ideals}. As $\ann(\ann(I))$ is a graded ideal, it cannot be equal to $I.$ This is a contradiction with (1).  

If (2) holds, then $E$ satisfies Condition (K), so each ideal of $L_K(E)$ is graded. Theorems \ref{theorem_char_of_reflexive} and \ref{theorem_Boolean_graded} imply that (3) holds. 

If (3) holds, the proof that (1) holds is the same as the proof of (2) $\Rightarrow$ (1) in Theorem \ref{theorem_Boolean_graded} (note that we did not use that the ideal is graded in the proof of that implication).
\end{proof}
 
\subsection{Quotient and porcupine graphs} 
Next, we show Proposition \ref{proposition_quotients} displaying some favorable features of being (strongly) all-reflexive which imply Corollaries \ref{corollary_quotients} and \ref{corollary_quotients_c_star}. 

\begin{proposition}
If $(H,S)$ is a reflexive admissible pair of $E,$ then $E$ is all-reflexive if and only if $E/(H,S)$ and $P_{(H,S)}$ are all-reflexive.  

If ``all-reflexive'' is replaced by ``strongly all-reflexive'', the statement continues to hold. 
\label{proposition_quotients} 
\end{proposition}
\begin{proof}
If $E$ is all-reflexive and $(H,S)$ is an admissible pair, then $S=B_H$ by Proposition \ref{proposition_on_reflexive}, so $E/(H,S)$ contains only the vertices and edges of $E.$ If $c$ is a cycle of $E/(H,S),$ then $c$ is a cycle of $E$ also and  $c^0\subseteq E^0-H.$ As $E$ is all-reflexive, $c$ is either without exits or extreme and, in either case, $T(c^0)$ is contained in $E^0-H.$ Hence, $c$ is either without exits or extreme in $E/(H,S).$ 

If $v$ is an infinite emitter of $E/(H,S),$ then $v\in E^0-H$ and, as $E$ is all-reflexive, $v$ is on a cycle which is necessarily extreme ($v$ emits more than one edge). This implies that $T(v)$ is in $E^0-H,$ so $v$ is on a cycle in $E/(H,S).$ 

If $\alpha$ is an infinite path of $E/(H,S),$ then $\alpha^0\subseteq E^0-H.$ As the set of bifurcations of $\alpha$ in $E/(H,S)$ is contained in the set of bifurcations of $\alpha$ in $E,$ we have that condition (2c) of Theorem \ref{theorem_char_of_reflexive} holds for $\alpha$ in $E/(H,S)$ because it holds for $\alpha$ in $E.$ This shows that $E/(H,S)$ is all-reflexive.

The porcupine graph $P_{(H,S)}$ has no new cycles, infinite emitters, nor infinite paths. If $E$ is all-reflexive, every cycle of $P_{(H,S)}$ is extreme or without exits because it is such in $E$ and every infinite emitter of $P_{(H,S)}$ is on an extreme cycle (in fact, $S=\emptyset$ by Lemma \ref{lemma_on_RH_bez_H}, so the edges of $P_{(H,S)}$ which are also in $E$
are only the edges of $E$ originating at vertices of $H$). If $\alpha$ is an infinite path of $P_{(H,S)}$, then $T(\alpha^0)\subseteq H$ which implies that condition (2c) of Theorem \ref{theorem_char_of_reflexive} holds for $\alpha$ in $P_{(H,S)}$ because it holds for $\alpha$ in $E.$ This shows that $P_{(H,S)}$ is all-reflexive.

Next, let us assume that $(H,S)$ is reflexive and that $P_{(H,S)}$ and $E/(H,S)$ are all-reflexive and let us show that $E$ is all-reflexive. If $c$ is a cycle of $E,$ then either $v\in H$ for some $v\in c^0$ or $c^0\subseteq E^0-H.$ In the first case, every vertex of $c$ is in $H$ and, as $H$ is hereditary, $T(c^0)\subseteq H.$ Thus, $c$ is either without exits or extreme in $P_{(H,S)}$ and so $c$ is such also in $E.$ In the second case, $c$ is a cycle in $E/(H,S).$ We claim that if $c$ has exits, then their ranges are in $E^0-H.$ Assume, on the contrary, that the range of an exit from $c$ is in $H.$ Then $c^0\subseteq R(H)-H,$ so $c^0\subseteq R(H^\bot)$ by Proposition \ref{proposition_on_reflexive}. As $c^0\subseteq R(H)$ implies that no vertex of $c$ is in $H^\bot,$ there has to be an exit from $c$ towards $H^\bot$ and the range of such exit is necessarily in $E^0-H$ since $H$ is hereditary. Thus, such exit is on a closed path returning to $c^0.$ As $H^\bot$ is hereditary, this implies that $c^0\subseteq H^\bot$ which is a contradiction with $c^0\subseteq R(H).$ This shows that the exits of $c$ have ranges necessarily in $E^0-H.$ As $E/(H,S)$ is all-reflexive, $c$ is without exits or extreme in $E$ because it is such in $E/(H,S).$ 

Let $v$ be an infinite emitter of $E.$ If $v$ does not emit edges to $E^0-H,$ then $v\in R(H)$ and $v\notin R(H^\bot).$ As $(H,S)$ is reflexive, this forces $v$ to be in $H,$ so $v$ is an infinite emitter in $P_{(H,S)}.$ Thus, $v$ is on a cycle in $P_{(H,S)}$ and, hence, also on a cycle in $E.$ If $v$ emits nonzero and finitely many edges to $E^0-H,$ then $v\in B_H.$ As $(H,S)$ is reflexive, $S=B_H$ and so $v$ is an infinite emitter of $P_{(H,S)}.$ Thus, $v$ is on a cycle in $P_{(H,S)}$ and, hence, also on a cycle in $E.$ If $v$ emits infinitely many edges to $E^0-H,$ then $v$ is an infinite emitter in $E/(H,S).$ So, it is on a cycle in $E/(H,S)$ and, hence, also on a cycle in $E.$

If $\alpha$ is an infinite path of $E,$ then either $v\in H$ for some $v\in\alpha^0$ or $\alpha^0\subseteq E^0-H.$ In the first case, let $v$ be the first vertex of $\alpha$ which is in $H$ and let $\alpha=pv\beta$ for some finite path $p$ and an infinite path $\beta.$ As $P_{(H,S)}$ is all-reflexive, only finitely many ranges of the bifurcations from $\beta$ can be in $E^0-R(\alpha^0).$ In addition, no vertex of $p$ is an infinite emitter, because such a vertex would be on an extreme cycle of $E^0-H$ and we would have that $\beta^0\subseteq E^0-H.$ However, $\beta^0\subseteq H.$ This shows that only finitely many bifurcations from $\alpha$ can have ranges outside of $R(\alpha^0).$

In the second case, $\alpha$ is an infinite path of $E/(H,S).$ If $\alpha$ has infinitely many bifurcations, only finitely many of them which have ranges in $E^0-H$ can be outside of $R(\alpha^0).$ 
Thus, $\alpha$ could possibly have infinitely many bifurcations with ranges not in $R(\alpha^0)$ only if $\alpha$ has infinitely many bifurcations to $H.$ We claim that this cannot happen. Assume, on the contrary, that it does happen. If a vertex of $\alpha$ is on a closed path, then that closed path has an exit towards $H.$ Since the cycles of $E$ which have exits are extreme and $H$ is hereditary, that would imply that all subsequent vertices of $\alpha$ are in $H.$ As $\alpha^0\subseteq E^0-H,$ this shows that no vertex of $\alpha$ is on a closed path. Hence, $\alpha$ is strictly decreasing and for every $v\in \alpha^0$ there is $w\in \alpha^0$ such that $v\gneq w$ (recall that this means $R(v)\subsetneq R(w)$) and $w$ emits an edge to $H.$ As every such $w$ is in $R(H)-H\subseteq R(H^\bot)$ and $w\notin E^0-R(H)=H^\bot,$ such $w$ emits a path which ends at $H^\bot$ and, hence, such path departs $\alpha$ at some point. Thus, the existence of infinitely many bifurcations to $H$ implies the existence of infinitely many bifurcations also towards $H^\bot.$  

Let $v_0\gneq v_1\gneq \ldots$ be the vertices of $\alpha$ such that $v_n$ emits an edge to $H$ and a path $p_n$ to $H^\bot$ for every $n.$ As $p_n$ departs $\alpha$ eventually and $E/(H, S)$ is all-reflexive, there is a vertex on $p_n$ which is not in $\alpha^0$ and which is in $R(\alpha^0)$ for infinitely many $n.$ Considering only those $n$ and $v_n$ for which this is the case, we can assume that this happens for every $n.$

Let $w_n\in R(\alpha^0)-\alpha^0$ be the vertex on $p_n$ such that all subsequent vertices of $p_n$ are not in $R(\alpha^0).$ As $\ra(p_n)\in H^\bot$ and $w_n\in R(H), w_n\notin H^\bot$ so $\ra(p_n)\neq w_n.$ Thus, let $e_n$ be an edge of $p_n$ with the source $w_n.$ Since $w_n\in R(\alpha^0),$ $w_n$ emits a path $q_n$ with the range in $\alpha^0$ and, as $\ra(e_n)\notin R(\alpha^0),$ the first edge of $q_n$ is not on $p_n.$ Let $n_0=0$ and let $\{v_{n_m}\mid m=0,1,\ldots\}$ be a subsequence of $\{v_n\mid n=0,1,\ldots\}$ such that $v_{n_{m+1}}$ is strictly after $\ra(q_{n_m})$ on $\alpha$ for all $m=0,1,\ldots.$ Let $r_{n_m}$ denote the part of $p_{n_m}$ from $\so(p_{n_m})=v_{n_m}$ to $w_{n_m}$ and let $s_{n_m}$ denote the part of $\alpha$ from $\ra(q_{n_m})$ to $v_{n_{m+1}}.$ Consider the infinite path $\beta=r_0q_0s_0r_{n_1}q_{n_1}s_{n_1}r_{n_2}q_{n_2}s_{n_2}\ldots.$ This is a strictly decreasing path of $E^0-H$ with infinitely many bifurcations $e_0,e_{n_1},e_{n_2}\ldots$ such that the range of $e_{n_m}$ is not in $R(\alpha^0)$ so it is not in $R(\beta^0)$ (note that $R(\beta^0)\subseteq R(\alpha^0)$). This contradicts the assumption that $E/(H,S)$ is all-reflexive and finishes the proof of the claim that $E$ is all-reflexive. 

If $E$ is strongly all-reflexive, then $E/(H,S)$ is such also because it cannot happen that an extreme cycle of $E$ becomes without exits in $E/(H,S).$ In addition, $P_{(H,S)}$ is also strongly all-reflexive since all cycles of $P_{(H,S)}$ are cycles of $E.$
Conversely, if $P_{(H,S)}$ and $E/(H,S)$ are strongly all-reflexive, and $c$ is a cycle of $E$, 
then $c$ is a cycle either in $P_{(H,S)}$ or in $E/(H,S).$ Since these graphs are strongly all-reflexive, $c$ has exits. So, $E$ is strongly all-reflexive.  
\end{proof}

If an admissible pair $(H,S)$ is not reflexive, it is possible to have that $E/(H,S)$ and $P_{(H,S)}$ are strongly all-reflexive and $E$ is not all-reflexive as the following example shows. 

\begin{example}
Let $E$ be the first graph below and let $H=\{v\}.$ As $H^\bot=\emptyset$ and $R(H)-H=\{u\},$ $H$ is not reflexive, so $E$ is not all-reflexive. The second graph is the quotient graph $E/H$ and the third graph is the porcupine graph $P_H.$ The use of the dotted lines in $P_H$ indicates that the pattern that every vertex emits one and receives two edges continues after the first row. Both $E/H$ and $P_H$ are strongly all-reflexive.  
{\small
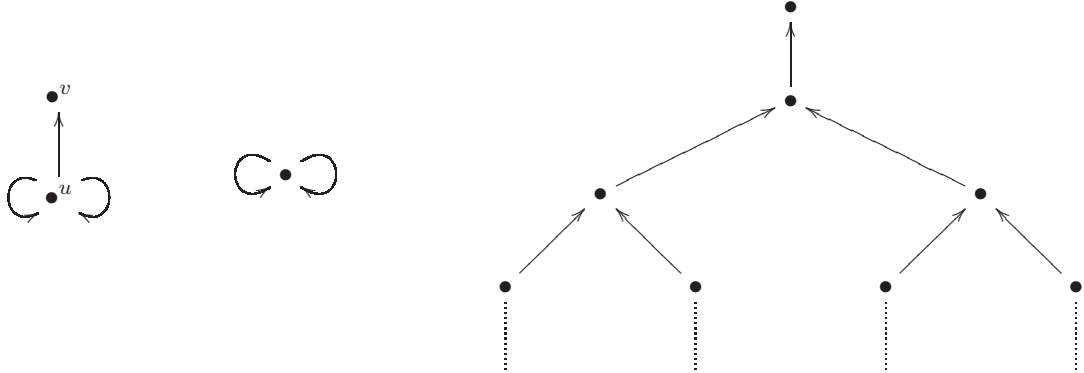
\[\xymatrix{\\\bullet^v\\\bullet^u\ar@(lu,ld)\ar@(ru, rd)  \ar[u] }\hskip2.5cm
\xymatrix{\\\\\bullet\ar@(lu,ld)\ar@(ru, rd)}\hskip2.5cm 
\xymatrix{
&&&\bullet&&&\\
&&&\bullet\ar[u]&&&\\
&\bullet\ar[urr]&&&&\bullet\ar[ull]&\\
\bullet\ar[ur]\ar@{.}[d]&&\bullet\ar[ul]\ar@{.}[d]&&\bullet\ar[ur]\ar@{.}[d]&&\bullet\ar[ul]\ar@{.}[d]\\
&&&&&&\\
}
\]}

As another example, let $E$ be the first graph below and let $H$ be the union of its sinks. As Example \ref{example_not_reflexive} shows, $H$ is not reflexive, so $E$ is not all-reflexive. 
The second graph is the quotient graph $E/H$ and the third is the porcupine graph $P_H.$ Both $E/H$ and $P_H$ are strongly all-reflexive. 
{\small
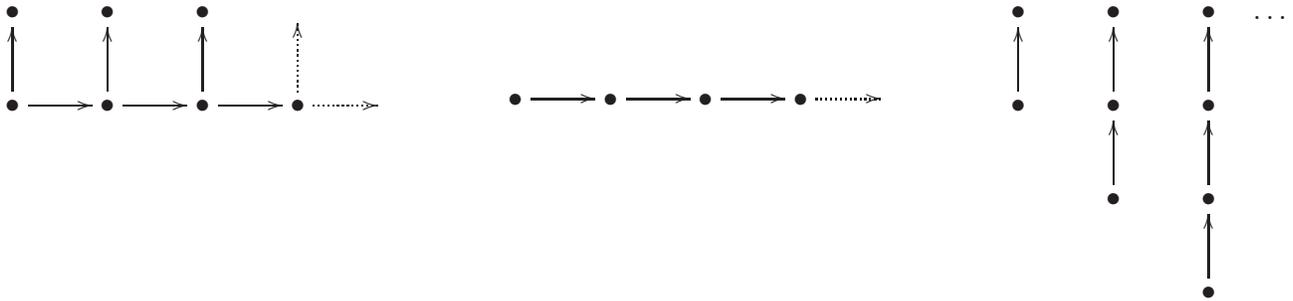
\[\xymatrix{  \bullet     & \bullet     & \bullet  &   & \\   \bullet \ar[r]\ar[u] & \bullet \ar[r]  \ar[u] & \bullet \ar[r]\ar[u] & \bullet \ar@{.>}[r] \ar@{.>}[u] &}\hskip1.4cm  \xymatrix{ \\ \bullet  \ar[r] & \bullet  \ar[r] & \bullet  \ar[r]& \bullet\ar@{.>}[r]&}\hskip1.4cm
\xymatrix{ \bullet        & \bullet       & \bullet \\ \bullet \ar[u] & \bullet\ar[u] & \bullet\ar[u] \\ & \bullet\ar[u] & \bullet \ar[u] \\  &  & \bullet \ar[u]}\;\;\;\ldots  \]}  
\label{example_porcupine_and_quotients} 
\end{example}
Proposition \ref{proposition_quotients} has an interesting corollary we present next. The proof follows directly from Proposition \ref{proposition_quotients} and Theorems \ref{theorem_char_of_reflexive} and \ref{theorem_Boolean}.  

\begin{corollary}
If $I$ is an annihilator ideal of $L_K(E),$ then  each graded ideal of $L_K(E)$ is an annihilator ideal if and only if each graded ideal of $I$ and each graded ideal of $L_K(E)/I$ are annihilator ideals. 

If ``graded ideal'' is replaced by ``ideal'' in the above statement, the statement continues to hold.     
\label{corollary_quotients}
\end{corollary}

\section{Annihilator ideals of graph \texorpdfstring{$C^*$}{TEXT}-algebras}
\label{section_C_star}

\subsection{Graph \texorpdfstring{$\mathbf{C^*}$}{TEXT}-algebras}\label{subsection_intro_C_star}
If $E$ is a graph, the {\em graph $C^*$-algebra of $E$} is the universal $C^*$-algebra generated by mutually orthogonal projections $\{p_v\mid v\in E^0\}$ and partial isometries with mutually orthogonal ranges $\{s_e\mid e\in E^1\}$ satisfying the analogues of the (CK1) and (CK2) axioms and the axiom (CK3) stating that $s_es_e^*\leq p_{\so(e)}$ for every $e\in E^1$ (where $\leq$ is the order on the set of projections given by $p\leq q$ if $p=pq=qp$). The term ``universal'' in the definition means that the $C^*$-algebra version of the algebraic Universal Property  holds (see \cite[Definition 5.2.5]{LPA_book}). By letting $s_{e_1\ldots e_n}$ be $s_{e_1}\ldots s_{e_n}$ and $s_v=p_v$ for $e_1,\ldots ,e_n\in E^1$ and $v\in E^0,$ $s_p$ is defined for every path $p.$

The set $\{p_v, s_e\mid v\in E^0, e\in E^1\}$ is referred to as a {\em Cuntz-Krieger $E$-family}. For such an $E$-family and an element $z$ of the unit circle $\mathbb T$ in the complex plane, one defines a map $\gamma_z$ by $\gamma_z(p_v)=p_v$ and $\gamma_z(s_e)=zs_e$ and then uniquely extends this map to a $*$-automorphism of $C^*(E)$ (we assume a homomorphism of a $C^*$-algebra to be bounded). The {\em gauge action} $\gamma$ on $\mathbb T$ is given by $\gamma(z)=\gamma_z.$ 
A closed ideal $I$ of a graph $C^*$-algebra $C^*(E)$ is {\em gauge-invariant} if $\gamma_z(I)=I$ for every $z\in \mathbb T.$ By \cite[Theorem 3.6]{Bates_et_al}, each such ideal $I$ is the closure of the linear span of the elements $s_ps_q^*$ for paths $p,q$ with $\ra(p)=\ra(q)\in H$ and the elements $s_pp_v^Hs_q^*$ for paths $p,q$ with $\ra(p)=\ra(q)=v\in S$ where $p_v^H=p_v-\sum_{e\in \so^{-1}(v)\cap \ra^{-1}(E^0-H)}s_es_e^*$ for $v\in B_H$ and where $(H, S)$ is the admissible pair defined analogously as for a graded ideal of $L_K(E)$. An admissible pair $(H,S)$ uniquely determines a closed gauge-invariant ideal $I(H,S)$ and the lattice of closed gauge-invariant ideals of $C^*(E)$ is isomorphic to the lattice of admissible pairs and, hence, also to the lattice of graded ideals of $L_{\Cset}(E).$    

Let $E$ be a row-finite graph, let $I$ be a closed ideal of $C^*(E),$ and let $H=I\cap E^0.$ 
The ideal $J$ of $C^*(E/H)$ corresponding to $I/I(H)$ is contained in a closed ideal generated by the vertices of cycles without exits in $E/H$ (see \cite[Section 5.4]{LPA_book}). Let $C_H$ be the set of all cycles of $E$ such that they become without exits in $E/H$ and such that they have a nontrivial intersection with $J$ in $C^*(E/H).$ For every $c\in C_H,$ there is a finite-dimensional or separable infinite-dimensional Hilbert space $\mathcal H$ \footnote{The space $\mathcal H$ is finite-dimensional if the number of paths which end at a vertex of $c$ and which do not contain $c$ is finite. If this set of paths is infinite, $\mathcal H$ is separable and infinite-dimensional.} 
and a compact set $K_c\subseteq \mathbb T$ such that
the intersection of $J$ and the closed ideal generated by $c^0$ is $*$-isomorphic to $\mathcal K\otimes C_0(\mathbb T-K_c)$ where $\mathcal K$ is the algebra of compact operators on $\mathcal H$ and $C_0(\mathbb T-K_c)$ is the algebra of continuous functions which disappear at infinity (for a locally compact set $X\subseteq \Cset,$ $C_0(X)$ is the algebra of complex-valued continuous functions $f$ on $X$ such that for every $\varepsilon>0,$ there is a compact set $K\subseteq X$ such that $|f|<\varepsilon$ outside of $K$).

\subsection{Annihilator ideals of graph \texorpdfstring{$C^*$}{TEXT}-algebras}
Since each closed ideal of a $C^*$-algebra is self-adjoint, the subscripts can be dropped from $\ann_l$ and $\ann_r$ and it is sufficient to consider only the operator $\ann.$ While each annihilator ideal of $L_K(E)$ is graded, a graph $C^*$-algebra can have closed annihilator ideals which are not gauge-invariant (see \cite[Remark 3.12]{Goncalves_Royer_regular_ideals} and note that Lemma \ref{lemma_on_CT} has some more specifics), so the graph $C^*$-algebra version of Proposition \ref{annihilators_are_graded} does not hold.
On the other hand, the annihilator of a gauge-invariant ideal is gauge-invariant (see \cite[Lemma 3.2]{Brown_et_al}). Thus, the proof of Proposition \ref{annihilator_of_graded_ideal} directly adjusts to  the proof of the following corollary.  
\begin{corollary}
The annihilator of a closed gauge-invariant ideal $I(H,S)$ of $C^*(E)$ is $I(H^\bot, S^\bot).$
\label{annihilator_of_gauge_inv_ideal}
\end{corollary}

Proposition \ref{proposition_on_reflexive} holds for a closed gauge-invariant ideal of a graph $C^*$-algebra: conditions (2) and (3) are conditions on the graph only, so their equivalence is not impacted by whether we consider a Leavitt path algebra or a graph $C^*$-algebra and Corollary \ref{annihilator_of_gauge_inv_ideal} and Proposition \ref{closure_on_admissible} imply that the $C^*$-algebra version of (1) is equivalent to (2). 

Proposition \ref{proposition_condition_L} is a statement on $E$ only, not on $L_K(E)$ nor $C^*(E).$ Thus, Proposition \ref{proposition_condition_L} implies that \cite[Theorem 3.5, Proposition 3.7, and Corollary 3.8]{Brown_et_al} hold without requiring the graph to be row-finite. 

Next, we show the graph $C^*$-algebra analogue of Theorems \ref{theorem_char_of_reflexive} and \ref{theorem_Boolean_graded}.   

\begin{corollary}
The following conditions are equivalent for any graph $E.$  
\begin{enumerate}[\upshape(1)]
\item Each closed gauge-invariant ideal of $C^*(E)$ is an annihilator ideal. 

\item The graph $E$ is all-reflexive.  

\item The lattice of closed gauge-invariant ideals of $C^*(E)$ is a Boolean algebra. 
\end{enumerate}
\label{corollary_reflexive_c_star} 
\end{corollary}
\begin{proof}
The proof of Theorem \ref{theorem_char_of_reflexive}
shows that condition (2) holds if and only if $(H,S)$ is reflexive for every admissible pair $(H,S).$ As 
each closed gauge-invariant ideal $I$ of $C^*(E)$ is of the form $I=I(H,S)$ for some admissible pair $(H,S),$ and $I$ is an annihilator ideal if and only if $(H,S)$ is reflexive, this shows the equivalence of (1) and (2). 

Since the lattice of graded ideals of $L_{\Cset}(E)$ and the lattice of closed gauge-invariant ideals of $C^*(E)$ are isomorphic, one lattice is a Boolean algebra if and only if the other one is a Boolean algebra. Thus, conditions (1) and (2) are equivalent to condition (3) by Theorem \ref{theorem_Boolean_graded}.
\end{proof}

If $E$ is $\xymatrix{\bullet^v\ar@(ur, dr)^e},$ then  $C^*(E)$ is $*$-isomorphic to the algebra $C(\mathbb T)$ of continuous $\Cset$-valued functions on $\mathbb T.$ By \cite[Theorem 3.4.1]{Kadison_Ringrose}, each closed ideal $I$ of $C(\mathbb T)$ is uniquely determined by a closed set $K\subseteq \mathbb T$ such that every element of $I$ vanishes on $K$ and we write $I=I(K)$ in this case. We use this notation in the following lemma, needed for Theorem \ref{theorem_Boolean_c_star}. 

\begin{lemma} Statements (1) to (3) hold for the algebra $C(\mathbb T)$ and they imply statement (4). 
\begin{enumerate}[\upshape(1)]
\item If $K$ is a closed subset of $\mathbb T,$ then 
$\ann(I(K))=I(\overline{\mathbb T-K}).$

\item If $K$ is a closed subset of $\mathbb T,$ then $I(K)$ is an annihilator ideal if and only if the interior of $K$ is nonempty or $K$ is empty. 

\item The algebra $C(\mathbb T)$ has a proper closed  ideal $I$ such that $\ann(I)$ is trivial.  

\item If $E$ is a row-finite graph with a single cycle $c$ and $c$ is without exits and such that every infinite path contains a vertex of $c,$ then $C^*(E)$ has a proper closed ideal $I$ such that $\ann(I)$ is trivial.   
\end{enumerate}
\label{lemma_on_CT}
\end{lemma}
\begin{proof}
To show (1), note that for every closed set $K\subseteq \mathbb T$ and for every $x\in \mathbb T-K,$ there is a function in $C(\mathbb T)$ which is zero on $K$ and nonzero at $x$ (see the proof of \cite[Theorem 3.4.1]{Kadison_Ringrose}). This implies that $\ann(I(K))$ consists of the elements of $C(\mathbb T)$ which vanish on the closure of $\mathbb T-K.$  

As (2) clearly holds if $K=\emptyset,$ let us assume that $K\neq\emptyset.$ By part (1), the condition that $I(K)$ is an annihilator is equivalent with the requirement that $\overline{\mathbb T-\overline{\mathbb T-K}}=K.$ Since $\overline{\mathbb T-\overline{\mathbb T-K}}$ is the closure of the interior of $K$, it is equal to $K$ if and only if the interior of $K$ is nonempty. 

To show (3), let $K=\{1\}.$ Since $K\neq\emptyset,$  $I(K)$ is proper. As $\overline{\mathbb T-K}=\mathbb T,$ $\ann(I(K))=I(\mathbb T)=\{0\}.$ 

If the assumptions of (4) hold, then $C^*(E)$ is $*$-isomorphic to $\mathcal K\otimes C(\mathbb T)$ where $\mathcal K$ is as in section \ref{subsection_intro_C_star} (see \cite[Proposition 5.4.2]{LPA_book}). If $I$ is a proper closed ideal of $C(\mathbb T)$ with the trivial annihilator, then $\mathcal K\otimes I$ is a proper closed ideal of an isomorphic copy of $C^*(E)$ and its annihilator is trivial.  
\end{proof}

We show the $C^*$-algebra version of Theorem \ref{theorem_Boolean} next. 

\begin{theorem}
The following conditions are equivalent for any graph $E.$  
\begin{enumerate}[\upshape(1)]
\item The graph $E$ is strongly all-reflexive. 

\item The lattice of closed ideals of $C^*(E)$ is a Boolean algebra. 

\item Each closed ideal of $C^*(E)$ is an annihilator ideal.

\item Each closed ideal of $C^*(E)$ is gauge-invariant and an annihilator ideal. 

\end{enumerate}
\label{theorem_Boolean_c_star} 
\end{theorem}
\begin{proof}
If (1) holds, then $E$ satisfies Condition (K), so each closed ideal of $C^*(E)$ is gauge-invariant (see \cite[Corollary 3.8]{Bates_et_al}). Corollary \ref{corollary_reflexive_c_star} implies that (2) and (4) hold. As (4) trivially implies (3), it is sufficient to show that (2) $\Rightarrow$ (3) and (3) $\Rightarrow$ (1) to prove the theorem.  

Assume that (2) holds and let $J$ be the complement of a closed ideal $I.$ The condition $I\cap J=\{0\}$ implies that $J\subseteq \ann(I).$ To show the converse, let $r\in \ann(I)$ and let $\{p_\lambda\}_{ \lambda\in \Lambda}$ be an approximate unit (so that $\lim_{\lambda\in \Lambda} p_\lambda r=\lim_{\lambda\in \Lambda} rp_\lambda=r).$ As $C^*(E)$ is the closure of $\ann(I)+J,$ for each $\lambda\in\Lambda,$ there is a  net $\{x_{\lambda\mu}+y_{\lambda\mu}\}_{\mu\in M_\lambda}$ converging to $p_\lambda$ and such that $x_{\lambda\mu}\in \ann(I)$ and $y_{\lambda\mu}\in J$ for $\mu\in M_\lambda.$  Hence, we have that
\[r=\lim_{\lambda\in \Lambda} rp_\lambda=\lim_{\lambda\in \Lambda} r(\lim_{\mu\in M_\lambda} x_{\lambda\mu}+y_{\lambda\mu})=\lim_{\lambda\in \Lambda} \lim_{\mu\in M_\lambda} rx_{\lambda\mu}+ry_{\lambda\mu}=\lim_{\lambda\in \Lambda} \lim_{\mu\in M_\lambda}ry_{\lambda\mu}.\]
As $ry_{\lambda\mu}\in J$ and $J$ is closed, this shows that $r\in J.$ This also shows that the complement of $J=\ann(I)$ is $\ann(\ann(I)).$ Hence,  $I=\ann(\ann(I)).$ Thus, (3) holds.  

Assume that (3) holds. In this case, condition (1) of Corollary \ref{corollary_reflexive_c_star} also holds, so $E$ is all-reflexive. Thus, to show (1), it remains to show that $E$ has no cycles without exits. Assume, on the contrary, that a cycle $c$ is without exits and let $H$ be the hereditary and saturated closure of $c^0.$ Since  
$\left(E/(H^\bot, B_{H^\bot})\right)^0=E^0-H^\bot=R(H)$ and $E$ is all-reflexive, $E/(H^\bot, B_{H^\bot})$ is row-finite by Lemma \ref{lemma_on_RH_bez_H}. If $d$ is a cycle with vertices in $R(H),$ then $d=c$ also since $E$ is all-reflexive. Thus, $c$ is the only cycle in the graph $E/(H^\bot, B_{H^\bot}).$ 

We claim that the assumptions of part (4) of Lemma \ref{lemma_on_CT} hold for $E/(H^\bot, B_{H^\bot}).$ To show this, it is sufficient to show that every infinite path of $E/(H^\bot, B_{H^\bot})$ contains a vertex of $c.$ Assume that $\alpha$ is an infinite path in $E/(H^\bot, B_{H^\bot})$ and no vertex of $\alpha$ is in $c^0.$ As $c$ is the only cycle of $E/(H^\bot, B_{H^\bot}),$ $\alpha$ is strictly decreasing and, as $\alpha^0\subseteq R(c^0)-c^0,$ with infinitely many bifurcations. We construct an infinite path similarly as in the proof of Proposition \ref{proposition_quotients} to arrive to a contradiction. 
Let $v_0\gneq v_1\gneq \ldots$ be the vertices of $\alpha$ such that, for every $n$, $v_n$ emits a path $p_n$ whose vertices are in $R(c^0)-c^0$ except the last one, $\ra(p_n),$ which is in $c^0.$ As $p_n$ departs $\alpha$ eventually and $E$ is all-reflexive,  there is a vertex on $p_n$ which is not in $\alpha^0$ and which is in $R(\alpha^0)$ for infinitely many $n.$ Considering only those $n$ and $v_n$ for which this is the case, we can assume that this happens for every $n.$ Let $w_n\in R(\alpha^0)-\alpha^0$ be the vertex on $p_n$ such that all subsequent vertices of $p_n$ are not in $R(\alpha^0).$ As $R(\alpha^0)\cap c^0=\emptyset$ and $\ra(p_n)\in c^0,$ $\ra(p_n)\neq w_n.$ Thus, let $e_n$ be an edge of $p_n$ with the source $w_n.$ Since $w_n\in R(\alpha^0),$ $w_n$ emits a path $q_n$ with the range in $\alpha^0$ and, as $\ra(e_n)\notin R(\alpha^0),$ the first edge of $q_n$ is not on $p_n.$ Let $n_0=0$ and let $\{v_{n_m}\mid m=0,1,\ldots\}$ be a subsequence of $\{v_n\mid n=0,1,\ldots\}$ chosen so that $v_{n_{m+1}}$ is strictly after $\ra(q_{n_m})$ on $\alpha$ for all $m=0,1,\ldots.$ Let $r_{n_m}$ denote the part of $p_{n_m}$ from $\so(p_{n_m})=v_{n_m}$ to $w_{n_m}$ and let $s_{n_m}$ denote the part of $\alpha$ from $\ra(q_{n_m})$ to $v_{n_{m+1}}.$ Consider the infinite path $\beta=r_0q_0s_0r_{n_1}q_{n_1}s_{n_1}r_{n_2}q_{n_2}s_{n_2}\ldots.$ This is a strictly decreasing path of $E/(H^\bot, B_{H^\bot})$ with infinitely many bifurcations $e_0,e_{n_1},e_{n_2}\ldots$ with ranges not in $R(\alpha^0)$ and, hence, not in $R(\beta^0).$ This contradicts the assumption that $E$ is all-reflexive and proves the claim.  
 
By part (4) of Lemma \ref{lemma_on_CT}, there is a proper closed ideal $J$ of $C^*(E/(H^\bot, B_{H^\bot}))$ with the trivial annihilator. Let $I$ be the proper closed ideal of 
$C^*(E)$ such that $I/I(H^\bot, B_{H^\bot})$ corresponds to $J$ under a $*$-isomorphism of $C^*(E)/I(H^\bot, B_{H^\bot})$ onto $C^*(E/(H^\bot, B_{H^\bot})).$ We claim that $\ann(I)=0.$ If $x\in \ann(I),$ then $x\in \ann(I(H^\bot, B_{H^\bot}))=I(H^{\bot\bot}, B_{H^\bot}^\bot)=I(H, B_H-B_{H^\bot})=I(H, B_H).$ On the other hand, $x\in \ann(I)$ implies that $x+I(H^\bot, B_{H^\bot})$ is in the annihilator of $I/I(H^\bot, B_{H^\bot})$ in $C^*(E)/I(H^\bot, B_{H^\bot}).$ Since this annihilator is trivial, $x$ is in $I(H^\bot, B_{H^\bot}).$ Thus, we have that $x$ is in the intersection $I(H, B_H)\cap I(H^\bot, B_{H^\bot}).$ As this intersection is trivial by Corollary \ref{corollary_reflexive_c_star} (note that $B_H^\bot=B_{H^\bot}$), we have that $x=0.$ Hence, $\ann(I)=0$ which implies that $I\subsetneq \ann(\ann(I))$ since $I$ is proper and $\ann(\ann(I)) =C^*(E).$  This is a contradiction with (3). Hence, no cycle $c$ without exits can exist. This shows that  showing that (1) holds. 
\end{proof} 

As Lemma \ref{lemma_on_CT} shows, there can be closed annihilator ideals which are not gauge-invariant in a graph $C^*$-algebra. By Theorem \ref{theorem_Boolean_c_star}, this cannot happen if {\em all} of the closed ideals are annihilators. 

Lastly, we state the following corollary of Proposition \ref{proposition_quotients}, Corollary \ref{corollary_reflexive_c_star} and Theorem \ref{theorem_Boolean_c_star}.  

\begin{corollary}
If $I$ is a closed gauge-invariant ideal of $C^*(E)$ which is an annihilator ideal, then each closed gauge-invariant ideal of $C^*(E)$ is an annihilator ideal if and only if each closed gauge-invariant ideal of $I$ and each closed gauge-invariant ideal of $C^*(E)/I$ are annihilator ideals.  

The statement remains true if each occurrence of 
``gauge-invariant'' except the first one is deleted.     
\label{corollary_quotients_c_star}
\end{corollary}


\end{document}